\documentclass[a4paper,12pt]{amsart}
\usepackage[hmarginratio={1:1},vmarginratio={1:1},heightrounded,textwidth=420pt]{geometry}

\usepackage{amsfonts,amsmath,amssymb,amsthm}
\usepackage{mathtools}
\usepackage{graphicx}
\usepackage{hyperref}
\usepackage{cleveref}
\usepackage[utf8]{inputenc}
\usepackage{url}
\usepackage{floatrow}
\usepackage{tabularx}
\usepackage{subcaption}
\newcolumntype{Y}{>{\centering\arraybackslash}X}
\let\leq\leqslant
\let\geq\geqslant

\let\rho\varrho

\newcommand{\brac}[1]{{\left(#1\right)}}
\newcommand{\sbrac}[1]{{\left[#1\right]}}

\newcommand{\floor}[1]{{\left\lfloor #1 \right\rfloor}}
\newcommand{\ceil}[1]{{\left\lceil #1 \right\rceil}}

\makeatletter
\newcommand\ie{i.e\@ifnextchar.{}{.\@}}
\newcommand\etc{etc\@ifnextchar.{}{.\@}}
\newcommand\etal{et~al\@ifnextchar.{}{.\@}}
\makeatother

\newtheorem{theorem}{Theorem}

\newtheorem{lemma}[theorem]{Lemma}

\newtheorem{example}[theorem]{Example}

\newtheorem{definition}[theorem]{Definition}

\begin{document}

\thispagestyle{empty}

\title{A new lower bound for the on-line coloring of intervals with bandwidth}

\author[P.~Mikos]{Patryk Mikos}
\thanks{Research partially supported by NCN grant number 2014/14/A/ST6/00138.}

\address{
Theoretical Computer Science Department,%
\\Faculty of Mathematics and Computer Science,%
\\Jagiellonian University, Krak\'ow, Poland%
}
\email{mikos@tcs.uj.edu.pl}

\maketitle

\begin{abstract}
The on-line interval coloring and its variants are important combinatorial problems with many applications in network multiplexing, resource allocation and job scheduling.
In this paper we present a new lower bound of $4.1626$ for the competitive ratio for the on-line coloring of intervals with bandwidth which improves the best known lower bound of $\frac{24}{7}$.
For the on-line coloring of unit intervals with bandwidth we improve the lower bound of $1.831$ to $2$.
\end{abstract}

\section{Introduction}

An \emph{on-line coloring of intervals with bandwidth} is a two-person game, played in rounds by Presenter and Algorithm.
In each round Presenter introduces a new interval on the real line and its bandwidth - a real number from $\sbrac{0,1}$.
Algorithm assigns a color to the incoming interval in such a way that for each color $\gamma$ and any point $p$ on the real line, the sum of bandwidths of intervals containing $p$ and colored $\gamma$ does not exceed $1$.
The color of the new interval is assigned before Presenter introduces the next interval and the assignment is irrevocable.
The goal of Algorithm is to minimize the number of different colors used during the game, while the goal of Presenter is to maximize it.

An \emph{on-line coloring of unit intervals with bandwidth} is a variant of on-line coloring of intervals with bandwidth game in which all introduced intervals are of length exactly $1$.

In the context of various on-line coloring games, the measure of quality of a strategy for Algorithm is given by the competitive analysis.
A coloring strategy for Algorithm is \emph{$r$-competitive} if it uses at most $r \cdot c$ colors for any $c$-colorable set of intervals.
The \emph{absolute competitive ratio} for a problem is the infimum of all values $r$ such that there exists an $r$-competitive strategy for Algorithm for this problem.
Let $\chi_{A}(\mathcal{I})$ be the number of colors used by Algorithm $A$ on the set $\mathcal{I}$ of intervals with bandwidth, and $OPT\brac{\mathcal{I}}$ be the minimum number of colors required to color intervals in the set $\mathcal{I}$.

The \emph{asymptotic competitive ratio} for Algorithm $A$, denoted by $\mathcal{R}_{A}^{\infty}$, is defined as follows:
$$\mathcal{R}_{A}^{\infty} = \liminf \limits_{k \rightarrow \infty} \textstyle \{\frac{\chi_{A}(\mathcal{I})}{k} : OPT\brac{\mathcal{I}} = k \}$$
The \emph{asymptotic competitive ratio} for a problem is the infimum of all values $\mathcal{R}_{A}^{\infty}$ such that $A$ is an Algorithm for this problem.

In this paper we give lower bounds on competitive ratios for on-line coloring of intervals with bandwidth and for unit version of this problem.
We obtain these results by presenting explicit strategies for Presenter that force Algorithm to use many colors while the presented set of intervals is colorable with a smaller number of colors.

\subsection{Previous work}

A variant of on-line coloring of intervals with bandwidth in which all intervals introduced by Presenter have bandwidth $1$ is known as an on-line interval coloring.
The competitive ratio for this problem was established by Kierstead and Trotter~\cite{KiersteadT81}.
They constructed a strategy for Algorithm that uses at most $3\omega - 2$ colors on $\omega$-colorable set of intervals.
They also presented a matching lower bound -- a strategy for Presenter that forces Algorithm to use at least $3\omega - 2$ colors.
The unit variant of the on-line interval coloring was studied by Epstein and Levy~\cite{EpsteinL05_2}.
They presented a strategy for Presenter that forces Algorithm to use at least $\floor{\frac{3\omega}{2}}$ colors.
Moreover, they showed that a natural greedy algorithm uses at most $2\omega-1$ colors.

A variant of the on-line coloring of intervals with bandwidth in which all intervals have the same endpoints is known as the on-line bin packing, see \cite{Csirik98} for a survey.

On-line coloring of intervals with bandwidth was first posed in 2004.
Adamy and Erlebach \cite{AdamyE04} showed a $195$-competitive algorithm for this problem.
An improved analysis by Pemmaraju \etal~\cite{PemmarajuRV11} showed that Adamy-Erlebach algorithm has competitive ratio $35$.
Narayanaswamy \cite{Narayanaswamy04} and Azar \etal~\cite{AzarFLN06} presented a $10$-competitive algorithm,
while Epstein and Levy \cite{EpsteinL05} showed a lower bound of $\frac{24}{7}$ for the asymptotic competitive ratio in this problem.
On-line coloring of unit intervals with bandwidth was studied by Epstein and Levy \cite{EpsteinL05}.
They presented a lower bound of $2$ and upper bound of $\frac{7}{2}$ for the absolute competitive ratio in this problem.
For the asymptotic competitive ratio, they showed a $3.178$-competitive algorithm and a lower bound of $1.831$.

\subsection{Our result}

For the on-line coloring of intervals with bandwidth, we prove that the asymptotic competitive ratio is at least $4.1626$.
For the on-line coloring of unit intervals with bandwidth, we present an explicit strategy for Presenter that forces Algorithm to use at least $2k-1$ different colors while the presented set of intervals is $k$-colorable.

\newpage
\section{Interval coloring}

At first we recall a strategy proposed by Kierstead and Trotter for Presenter in the on-line interval coloring game.
We use this strategy as a substrategy in our main result.

\begin{theorem}[Kierstead, Trotter \cite{KiersteadT81}]\label{thm:KT_lower}
For every $\omega \in \mathbb{N}_{+}$,
there is a strategy for Presenter that forces Algorithm to use at least $3\omega-2$ different colors in the on-line interval coloring game played on a $\omega$-colorable set of intervals.
Moreover, Presenter can play in such a way that every introduced interval is contained in a fixed real interval $[L,R]$.
\end{theorem}

\medskip
Below we present a strategy for Presenter in the on-line coloring of intervals with bandwidth.
For a fixed $k \in \mathbb{N}_{+}$, we ensure that at any point of the game, the set of intervals introduced by Presenter is $k$-colorable.

\begin{definition}
A pair of sequences $\brac{ \sbrac{j_1,\ldots,j_n}, \sbrac{x_1,\ldots,x_n} }$ such that $x_i \in \mathbb{N}_{+}$, $j_i|k$ and $\forall_{q<i}: j_q < j_i \leq \frac{1}{3}k$
is called a \emph{k-schema}.
\end{definition}

\medskip
Not every $k$-schema describes a strategy for Presenter.
Later we give additional conditions that a given $k$-schema has to satisfy to describe a valid strategy.

The strategy for Presenter based on a $k$-schema $\brac{ \sbrac{j_1,\ldots,j_n}, \sbrac{x_1,\ldots,x_n} }$ consists of 2 phases.
The first phase, called \emph{separation phase}, consists of $n$ subphases indexed $1,\ldots,n$.
Let  $\mathcal{M}$ be the set of \emph{marked intervals}, which initially is empty.
For each color $c$ used by Algorithm in the separation phase, the set $\mathcal{M}$ contains the first interval colored by Algorithm with $c$.

All intervals introduced by Presenter in the $i$-th subphase are contained in same region $[L_i,R_i]$, have length $s_i = \frac{1}{2}(R_i - L_i)$, and bandwidth $\frac{j_i}{k}$, see \Cref{Fig:Subphases}.
Let $l_i$ be the rightmost right endpoint of a non-marked interval introduced in the $i$-th subphase, or $l_i = L_i + s_i$ if such an interval does not exist.
Let $r_i$ be the leftmost right endpoint of a marked interval introduced in the $i$-th subphase, or $r_i = R_i$ if such an interval does not exist.
For the first subphase set $L_1 = 0$, $R_1 = 2$ and for the $i$-th subphase $L_i = l_{i-1}$ and $R_i = r_{i-1}$, see \Cref{Fig:Subphases}, where $l_{i-1}$ and $r_{i-1}$ are values of those variables after the end of the $\brac{i-1}$-subphase.

In the $i$-th subphase, Presenter introduces new intervals until gets exactly $x_i$ new colors.
Let $p_i = \frac{1}{2}(l_i + r_i)$.
A new interval introduced by Presenter has endpoints $I = [p_i - s_i, p_i]$.
If Algorithm colors $I$ with one of already used colors, then $l_i$ changes to $p_i$.
Otherwise, $r_i$ changes to $p_i$ and interval $I$ is marked.

Assume that Presenter constructs some coloring after each subphase.
Let $\Gamma_{i}$ be the number of colors used by Presenter in that coloring on intervals in the set $\mathcal{M}$ after $i$-th subphase, see \Cref{Fig:Optimum}.
In the second phase, called the \emph{final phase}, Presenter uses a strategy from \Cref{thm:KT_lower} in the real interval $[L_{n+1},R_{n+1}]$ for $\omega = k-\Gamma_{n}$.
Each interval introduced by Presenter in the final phase has bandwidth $1$.
This completes the description of a strategy for Presenter based on a $k$-schema.

\begin{figure}[H] \setlength{\unitlength}{0.09in}
\centering
\begin{picture}(51,18) 


\put(20,1.5){\line(1,0){26}}
\put(20,1){\line(0,1){1}}
\put(46,1){\line(0,1){1}}

\put(3,3){\line(1,0){26}}
\put(3,2.5){\line(0,1){1}}
\put(29,2.5){\line(0,1){1}}

\put(0.75,0.5){\dashbox{0.25}(0.01,4)}
\put(49.25,0.5){\dashbox{0.25}(0.01,4)}
\put(29,0.5){\dashbox{0.25}(0.01,4)}
\put(46,0.5){\dashbox{0.25}(0.01,4)}
\put(0,0){\dashbox{0.25}(51,7)}

\put(0.25,5){$L_{i-1}$}
\put(47.75,5){$R_{i-1}$}
\put(28.25,5){$l_{i-1}$}
\put(45,5){$r_{i-1}$}

\put(0.5,7.5){$(i-1)$-th subphase}
\put(46.25,1.5){*}


\put(34,9){\line(1,0){8}}
\put(34,8.5){\line(0,1){1}}
\put(42,8.5){\line(0,1){1}}

\put(32,10.5){\line(1,0){8}}
\put(32,10){\line(0,1){1}}
\put(40,10){\line(0,1){1}}

\put(30.5,12){\line(1,0){8}}
\put(30.5,11.5){\line(0,1){1}}
\put(38.5,11.5){\line(0,1){1}}

\put(31,13.5){\line(1,0){8}}
\put(31,13){\line(0,1){1}}
\put(39,13){\line(0,1){1}}

\put(29.5,8){\dashbox{0.25}(0.01,7)}
\put(45.5,8){\dashbox{0.25}(0.01,7)}
\put(40,8){\dashbox{0.25}(0.01,7)}
\put(39,8){\dashbox{0.25}(0.01,7)}
\put(28,7.5){\dashbox{0.25}(19,10)}

\put(28.5,15.5){$L_i$}
\put(44.5,15.5){$R_i$}
\put(38.5,15.5){$l_i$}
\put(39.5,15.5){$r_i$}

\put(29,18){$i$-th subphase}
\put(42.25,9){*}
\put(40.25,10.5){*}

\end{picture}
\caption{Intervals introduced in the $i$-th subphase in relation to the intervals introduced in the $\brac{i-1}$-th subphase. Marked intervals are marked with *.}
\label{Fig:Subphases}
\end{figure}
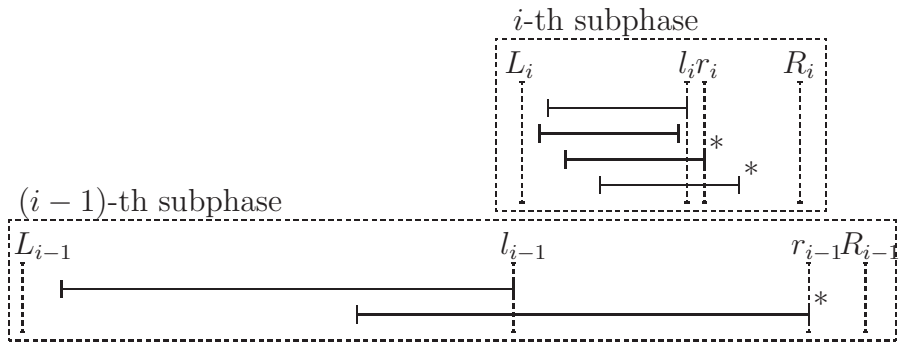

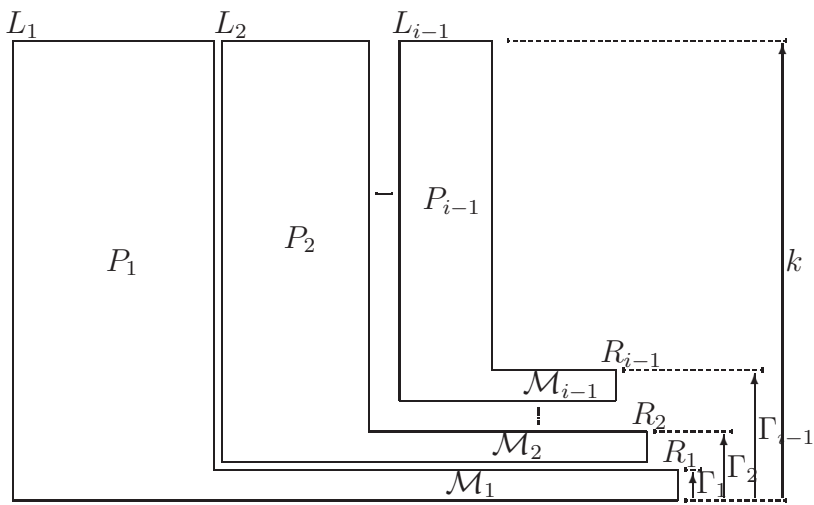
\begin{figure}[H] \setlength{\unitlength}{0.08in}
\centering
\begin{picture}(50,30) 


\put(0,0){\line(1,0){43}}
\put(43,0){\line(0,1){2}}
\put(43,2){\line(-1,0){30}}
\put(13,2){\line(0,1){28}}
\put(13,30){\line(-1,0){13}}
\put(0,30){\line(0,-1){30}}

\put(-0.5,30.5){$L_{1}$}
\put(42,2.5){$R_{1}$}
\put(6,15){$P_{1}$}
\put(28,0.5){$\mathcal{M}_{1}$}


\put(13.5,2.5){\line(1,0){27.5}}
\put(41,2.5){\line(0,1){2}}
\put(41,4.5){\line(-1,0){18}}
\put(23,4.5){\line(0,1){25.5}}
\put(23,30){\line(-1,0){9.5}}
\put(13.5,30){\line(0,-1){27.5}}

\put(13,30.5){$L_{2}$}
\put(40,5){$R_{2}$}
\put(17.5,16.5){$P_{2}$}
\put(31,3){$\mathcal{M}_{2}$}


\put(23.5,20){\dashbox{0.1}(1,0)}
\put(34,5){\dashbox{0.1}(0,1)}


\put(25,6.5){\line(1,0){14}}
\put(39,6.5){\line(0,1){2}}
\put(39,8.5){\line(-1,0){8}}
\put(31,8.5){\line(0,1){21.5}}
\put(31,30){\line(-1,0){6}}
\put(25,30){\line(0,-1){23.5}}

\put(24.5,30.5){$L_{i-1}$}
\put(38,9){$R_{i-1}$}
\put(26.5,19){$P_{i-1}$}
\put(33,7){$\mathcal{M}_{i-1}$}



\put(32,30){\dashbox{0.25}(18,0)}
\put(43.5,0){\dashbox{0.25}(6.5,0)}
\put(43.5,2){\dashbox{0.25}(1,0)}
\put(41.5,4.5){\dashbox{0.25}(5,0)}
\put(39.5,8.5){\dashbox{0.25}(9,0)}

\put(49.8,0.2){\vector(0,1){29.6}}
\put(50,15){$k$}

\put(44,0.2){\vector(0,1){1.6}}
\put(44.2,0.5){$\Gamma_{1}$}

\put(46,0.2){\vector(0,1){4.1}}
\put(46.2,1.5){$\Gamma_{2}$}

\put(48,0.2){\vector(0,1){8.1}}
\put(48.2,4){$\Gamma_{i-1}$}

\end{picture}
\caption{Distribution of intervals in the first $i-1$ separation subphases.}
\label{Fig:Optimum}
\end{figure}

Now, we show when a given $k$-schema $\brac{ \sbrac{j_1,\ldots,j_n}, \sbrac{x_1,\ldots,x_n} }$ actually describes a valid strategy for Presenter,
\ie{} when Presenter is able to force Algorithm to use at least $x_i$ new colors in the $i$-th subphase.

\medskip
During the $i$-th subphase, we have $L_i + s_i \leq p_i \leq R_i$.
Thus, the distance between rightmost right endpoint of an interval introduced in the $i$-th subphase and leftmost right endpoint of interval introduced in this phase is at most $s_i$, hence all intervals introduced by Presenter in the $i$-th subphase form a clique.
See \Cref{Fig:Subphases}.

Note that to the right of $l_{i-1}$ there are only marked intervals from subphases $1,\ldots,i-1$.
Each interval introduced in the $i$-th subphase intersects with every interval previously introduced in the $i$-th subphase and all marked intervals from subphases $1,\ldots,i-1$.
Thus, if Presenter introduces at most $\frac{k}{j_i}\brac{k - \Gamma_{i-1}}$ intervals in the $i$-th subphase, then all intervals can be colored with $k$ colors.

Intervals introduced in the $i$-th subphase intersect with exactly $\chi_{i-1}$ marked intervals from subphases $1,\ldots,i-1$ and by the definition of set $\mathcal{M}$ each of them has a different color.
Thus, in the real interval $[L_i,R_i]$, each color marked in a subphase $1 \leq q < i$ has accumulated bandwidth $\frac{j_q}{k}$.
We assumed that $\forall_{q<i}: j_q < j_i$, hence at most $\frac{k}{j_i}-1$ intervals can be colored with a color $c$ that was used in the previous subphases.
Thus, in the $i$-th subphase Algorithm can be forced to use at least $\Delta_i = \ceil{\frac{j_i}{k}\brac{ \frac{k}{j_i}\brac{k-\Gamma_{i-1}} - \chi_{i-1}\brac{\frac{k}{j_i}-1} }} = k - \Gamma_{i-1} - \chi_{i-1} + \ceil{\frac{j_i}{k}\chi_{i-1}}$ new colors.

After $i$-th subphase, Presenter assigns colors to the intervals introduced in the $i$-th subphase.
First, to the new marked intervals using greedy algorithm.
Then to the non-marked intervals from $i$-th subphase using at most $k$ colors in total.
Note that some intervals might be colored with some of already used $\Gamma_{i-1}$ colors.
Because the number of new marked intervals in the $i$-th subphase is exactly $x_i$, each of them has bandwidth $j_i$ and those intervals are colored using greedy algorithm, then $\Gamma_{i}$ is a properly defined quantity for a given $k$-schema and depends on $j_q$ and $x_q$ for all $q < i$, but does not depend on the Algorithm's coloring.

\begin{definition}
A $k$-schema $\brac{ \sbrac{j_1,\ldots,j_n}, \sbrac{x_1,\ldots,x_n} }$ is a \emph{$k$-strategy} if $\forall_{i}: x_i \leq \Delta_i$.
\end{definition}

Presenter using a $k$-strategy $\brac{ \sbrac{j_1,\ldots,j_n}, \sbrac{x_1,\ldots,x_n} }$ forces Algorithm to use at least  $\Sigma_{q=1}^{n}x_{q} + 3\brac{k-\Gamma_{n}}-2$ colors, while the set of introduced intervals is $k$-colorable.

\begin{example}
For a fixed $k \in \mathbb{N}_{+}$ and a $k$-strategy $\brac{ \sbrac{1}, \sbrac{k} }$ we have $\Gamma_{1} = 1$.
Presenter using this strategy forces Algorithm to use at least $k+3\brac{k-1}-2 = 4k-5$ colors, while the set of introduced intervals is $k$-colorable.
Thus, the asymptotic competitive ratio in the on-line coloring of intervals with bandwidth is at least $4$.
\end{example}

\begin{table}[H]
  \caption{Example of a strategy $S_{120}$.}
  \label{fig:Ex120}
  \begin{tabularx}{0.7\textwidth}{*{14}{|Y}|}\hline
    $j_i$                    & $1$   & $2$ & $3$ & $4$ & $5$ & $6$ & $8$ & $10$ & $12$ & $15$ & $20$ & $24$ & $30$  \\ \hline
    $x_i$                    & $120$ & $1$ & $1$ & $1$ & $1$ & $1$ & $2$ & $2$  & $2$  & $4$  & $5$  & $4$  & $8$   \\ \hline
    $\Gamma_{i}$             & $1$   & $2$ & $2$ & $2$ & $2$ & $2$ & $2$ & $2$  & $2$  & $3$  & $4$  & $4$  & $6$   \\ \hline
  \end{tabularx}
\end{table}

\begin{example}\label{Ex120}
Consider the $120$-strategy given by the values $j_i$, $x_i$ and $\Gamma_i$ from the \Cref{fig:Ex120}.
Presenter using this strategy forces Algorithm to use $152+3\brac{120-6}-2 = 492$ colors, while the set of introduced intervals is $120$-colorable.
Thus, the absolute competitive ratio for the on-line coloring of intervals with bandwidth is at least $4\frac{1}{10}$.
\end{example}

\Cref{Ex120} is an example of a $k$-strategy for Presenter for a fixed $k$, and gives a lower bound for the absolute competitive ratio.
In order to give lower bounds for the asymptotic competitive ratio, we introduce a notion of a \emph{scalable strategy}.

\newpage
\begin{definition}
For a $k$-strategy $S_k = \brac{ \sbrac{j_1,\ldots,j_n}, \sbrac{x_1,\ldots,x_n} }$,
an $ak$-schema $S_{k}^{a} = \brac{ \sbrac{aj_1,\ldots,aj_n}, \sbrac{ax_1,\ldots,ax_n} }$ for $a \in \mathbb{N}_{+}$ is called an $a$-scaled $S_k$ schema.
\end{definition}

Note that $a$-scaled $k$-strategy might not be an $ak$-strategy.
For example $S_{120}$ is a $120$-strategy but $S_{120}^{3}$ is not a $360$-strategy.
To see this, observe that in $S_{120}^{3}$ we have $12 = x_{10} > \Delta_{10} = 11$.

\medskip
Consider a $zk$-scaled $S_{k}$ schema for $z \in \mathbb{N}_{+}$.
We would like to introduce additional constraints on a $k$-strategy $S_k$ that will ensure that $S^{zk}_{k}$ is a $zk^2$-strategy.

\medskip
In the $i$-th subphase of the game played using a schema $S_{k}^{zk}$ Presenter forces Algorithm to use $zkx_i$ new colors.
Each new marked interval has bandwidth $\frac{zkj_i}{zk^2}=\frac{j_i}{k}$ for some $j_i$ such that $j_i|k$.
Thus, these new intervals after $i$-th subphase are colored by Presenter with $\frac{j_i}{k}zkx_i = zj_{i}x_{i}$ colors.
By induction, for the $S_{k}^{zk}$ schema we have $\Gamma_{i} = z\Sigma_{q=1}^{i}j_{q}x_{q}$ and $\chi_{i} = zk\Sigma_{q=1}^{i}x_{q}$.
Presenter using this schema in the $i$-th subphase can introduce $\frac{k}{j_i}\brac{zk^2 - \Gamma_{i-1}} = \frac{zk}{j_i}\brac{k^2 - \Sigma_{q=1}^{i-1}j_{q}x_{q}}$ intervals.
Moreover, at most $\frac{k}{j_i}-1$ intervals can be colored by Algorithm with an already used color $c$.
Thus, in the $i$-th subphase Algorithm is forced to use at least $\Delta_{i} = \ceil{\frac{j_i}{k} \brac{ \frac{zk}{j_i}\brac{k^2 - \Sigma_{q=1}^{i-1}j_{q}x_{q}} -  zk\Sigma_{q=1}^{i-1}x_{q}\brac{\frac{k}{j_i}-1} }}$ new colors, which after simplifying is $\Delta_{i} = z\brac{ k^2 + \Sigma_{q=1}^{i-1}\brac{j_i - j_q - k}x_{q} }$.
The $S_{k}^{zk}$ schema is a $zk^2$-strategy if $\forall_{i}: zkx_{i} \leq \Delta_{i}$.
Thus, we have a condition

\begin{equation}
\label{Eq:scalable}
x_{i} \leq k + \frac{1}{k}\Sigma_{q=1}^{i-1}\brac{j_i - j_q - k}x_{q}
\end{equation}

\begin{definition}
A $k$-strategy $S_k$ that satisfies \Cref{Eq:scalable} is called a \emph{scalable strategy}.
\end{definition}

Note that \Cref{Eq:scalable} does not depend on $z$.
This leads to the following lemma.

\begin{lemma}
If a $k$-strategy $S_k$ satisfies \Cref{Eq:scalable}, then for every $z \in \mathbb{N}_{+}$ a $zk$-scaled $S_{k}$ schema is a $zk^2$-strategy.
\end{lemma}

\begin{lemma}
For every $k,z \in \mathbb{N}_{+}$ and a scalable $k$-strategy $S_{k}$ the competitive ratio guaranteed by the $S_{k}^{zk}$ strategy is not less than the competitive ratio guaranteed by the $S_{k}$ strategy.
\end{lemma}

\begin{proof}
Let $S_k = \brac{ \sbrac{j_1,\ldots,j_n}, \sbrac{x_1,\ldots,x_n} }$.
Presenter using a strategy $S_k$ forces Algorithm to use $X = \Sigma_{i=1}^{n}x_i + 3\brac{k-\Gamma_{n}} - 2$ colors, while the set of intervals introduced by Presenter is $k$-colorable.
Presenter using a strategy $S_{k}^{zk}$ forces Algorithm to use $\bar{X} = \Sigma_{i=1}^{n}zkx_i + 3\brac{zk^2-\bar{\Gamma}_{n}} - 2$ colors, while the set of intervals introduced by Presenter is $zk^2$-colorable.
Observe that the number of colors required in greedy coloring of all intervals marked by $S_{k}^{zk}$ strategy is at most $zk$ times bigger than the number of colors required in greedy coloring of all intervals marked by $S_{k}$ strategy, \ie{} $\bar{\Gamma}_{n} \leq zk\Gamma_{n}$.
Thus, we have $\frac{1}{zk^2}\bar{X} \geq \frac{1}{k}X + \frac{2}{k} - \frac{2}{zk^2} > \frac{1}{k}X$.

\end{proof}

\begin{table}[H]
  \caption{Example of a scalable strategy $\bar{S}_{120}$}
  \label{fig:Ex120Scalable}
  \begin{tabularx}{0.7\textwidth}{*{18}{|Y}|}\hline
    $j_i$                    & $1$   & $2$ & $3$ & $4$ & $5$ & $6$ & $8$ & $10$ & $12$ & $15$ & $20$ & $24$ & $30$  \\ \hline
    $x_i$                    & $120$ & $1$ & $1$ & $1$ & $1$ & $1$ & $2$ & $2$  & $2$  & $3$  & $6$  & $4$  & $8$   \\ \hline
    $\Gamma_{i}$             & $1$   & $2$ & $2$ & $2$ & $2$ & $2$ & $2$ & $2$  & $2$  & $3$  & $4$  & $4$  & $6$   \\ \hline
  \end{tabularx}
\end{table}

\begin{example}
\Cref{fig:Ex120Scalable} is a description of a scalable strategy with competitive ratio $4\frac{1}{10}$.
This strategy implies a lower bound of $4\frac{1}{10} + \frac{2}{120} = 4\frac{7}{60}$ for the asymptotic competitive ratio in the on-line coloring of intervals with bandwidth.
\end{example}

In order to obtain the best lower bound for the asymptotic competitive ratio we can chose $k$ to be a \emph{highly composed number}.
As a sequence $j_1,\ldots,j_n$ we chose consecutive divisors of $k$ and as a sequence $x_1,\ldots,x_n$ we greedily choose the maximum numbers $x_i$ such that the resulting strategy is a scalable strategy.

\medskip
\Cref{fig:ListOfCoefs} contains the list of lower bounds for the asymptotic competitive ratio we got for some values of $k$ using this method.

\begin{table}[H]
  \caption{A table of asymptotic competitive ratios for different values of $k$}
  \label{fig:ListOfCoefs}
  \begin{tabular}{ | l | c | }\hline
    $k$  &  ratio  \\ \hline
    $60$  &  $4.0500000$ \\ \hline
    $120$  &  $4.1166667$ \\ \hline
    $360$  &  $4.1416667$ \\ \hline
    $840$  &  $4.1523809$ \\ \hline
    $2520$  &  $4.1587301$ \\ \hline
    $7560$  &  $4.1607142$ \\ \hline
    $10080$  &  $4.1611111$ \\ \hline
    $15120$  &  $4.1614417$ \\ \hline
    $25200$  &  $4.1615873$ \\ \hline
    $27720$  &  $4.1618326$ \\ \hline
    $110880$  &  $4.1621753$ \\ \hline
    $554400$  &  $4.1622763$ \\ \hline
  \end{tabular}
  \begin{tabular}{ | l | c | }\hline
    $k$  &  ratio  \\ \hline
    $2162160$  &  $4.1624500$ \\ \hline
    $21621600$  &  $4.1624777$ \\ \hline
    $183783600$  &  $4.1625239$ \\ \hline
    $2327925600$  &  $4.1625617$ \\ \hline
    $48886437600$  &  $4.1625717$ \\ \hline
    $321253732800$  &  $4.1625883$ \\ \hline
    $4497552259200$  &  $4.1625893$ \\ \hline
    $97821761637600$  &  $4.1625961$ \\ \hline
    $866421317361600$  &  $4.1626008$ \\ \hline
    $4043299481020800$  &  $4.1626015$ \\ \hline
    $12129898443062400$  &  $4.1626018$ \\ \hline
    $224403121196654400$  &  $4.1626043$ \\ \hline
  \end{tabular}
\end{table}

\begin{theorem}
Asymptotic competitive ratio for the on-line coloring of intervals with bandwidth is at least $4.1626$.
\end{theorem}

\newpage
\section{Unit intervals coloring}

\begin{theorem}
For every $k \in \mathbb{N}_{+}$, there is a strategy for Presenter that forces Algorithm to use at least $2k - 1$
different colors in the on-line coloring of unit intervals with bandwidth played on a $k$-colorable set of intervals.
\end{theorem}

\begin{proof}
For a given $k \in \mathbb{N}_{+}$, Presenter at first plays only the separation phase of a $k$-strategy $\brac{ \sbrac{1}, \sbrac{k} }$.
Because $L_1 = 0$, $R_1 = 2$ and $s_1 = \frac{1}{2}(R_1-L_1)$, every introduced interval has length $1$.
Moreover, there is a point $p_1 = \frac{1}{2}(l_1+r_1)$ such that every marked interval has its right endpoint to the right of $p_1$ and every non-marked interval has its right endpoint to the left of $p_1$.

Now, Presenter introduces $k - 1$ intervals $[p_1,p_1+1]$ of bandwidth $1$ each.
Every interval introduced in this phase gets a new color.
Thus, Algorithm uses $|\mathcal{M}| + k - 1 = 2k - 1$ colors in total, while the introduced set of intervals is $k$-colorable.

\end{proof}

\bibliographystyle{plainurl}
\bibliography{paper}
\end{document}